\documentclass[12pt,reqno]{amsart}

\usepackage{times}
\usepackage[T1]{fontenc}
\usepackage{mathrsfs}
\usepackage{latexsym}
\usepackage[dvips]{graphics}
\usepackage{epsfig}
\usepackage{verbatim}

\usepackage{amsmath,amsfonts,amsthm,amssymb,amscd}
\input amssym.def
\input amssym.tex
\usepackage{color}
\usepackage{hyperref}
\usepackage{color}
\usepackage{breakurl}
\usepackage{comment}
\newcommand{\bburl}[1]{\textcolor{blue}{\url{#1}}}

\newcommand{\be}{\begin{equation}}
\newcommand{\ee}{\end{equation}}
\newcommand{\bea}{\begin{eqnarray}}
\newcommand{\eea}{\end{eqnarray}}

\newtheorem{thm}{Theorem}[section]
\newtheorem{prblm}{Problem}[section]
\newtheorem{conj}[thm]{Conjecture}
\newtheorem{cor}[thm]{Corollary}
\newtheorem{lem}[thm]{Lemma}
\newtheorem{prop}[thm]{Proposition}

\newtheorem{prob}[thm]{Problem}

\newcommand{\Z}{\ensuremath{\mathbb{Z}}}

\newcommand{\N}{\mathbb{N}}
\newcommand{\F}{\mathbb{F}}

\numberwithin{equation}{section}

\usepackage{lastpage}
\begin{document}

\title{Modified Erd\H{o}s-Ginzburg-Ziv Constants for $(\Z/n\Z)^2$}

\author{Trajan Hammonds}
\address{Department of Mathematics, Carnegie Mellon University, Pittsburgh, PA 15213}
\email{\textcolor{blue}{\href{mailto:thammond@andrew.cmu.edu}{thammond@andrew.cmu.edu}}}
\maketitle
\date{\today}

\begin{abstract} 
For an abelian group $G$ and an integer $t > 0$, the \emph{modified Erd\H{o}s-Ginzburg-Ziv constant $s'_t(G)$} is the smallest integer $\ell$ such that any zero-sum sequence of length at least $\ell$ with elements in $G$ contains a zero-sum subsequence (not necessarily consecutive) of length $t$. We compute bounds for  $s'_{t}(G)$ for $G = \left(\Z/n\Z\right)^2$ and $G = \left(\Z/n_1\Z \times \Z/n_2\Z\right)$.  We also compute bounds for $G = \left(\Z/p\Z\right)^d$ where the subsequence can be any length in $\{p, \dots, (d-1)p\}$. Lastly, we investigate the Erd\H{o}s-Ginzburg-Ziv constant for $G = \left(\Z/n\Z\right)^2$ and subsequences of length $tn$.
\end{abstract}

\section{Introduction}
In 1961, Erd\H{o}s, Ginzburg, and Ziv proved the following theorem. 
\begin{thm}[Erd\H{o}s-Ginzburg-Ziv \cite{EGZ}]
Any sequence of length $2n-1$ in $\Z/n\Z$ contains a zero-sum subsequence of length $n$. 
\end{thm}
Many different proofs of this theorem have been given since the original in 1961. Perhaps the simplest proof makes use of the Chevalley-Warning theorem. Here, we don't require the subsequence to be consecutive, and a sequence is $\emph{zero-sum}$ if its elements sum to zero. This theorem has inspired many follow-up questions on zero-sum sequences. 

For the general case we consider the following problem: Let $G$ be an abelian group and let $\mathcal{L} \subseteq \N$. Then $s_\mathcal{L}(G)$ is defined to be the minimal $\ell$ such that any sequence of length $\ell$ with elements in $G$ contains a zero-sum subsequence whose length is in $\mathcal{L}$. When $\mathcal{L} = \mathrm{exp}(G)$, this is the Erd\H{o}s-Ginzburg-Ziv constant.

In this paper we will also study the \emph{modified Erd\H{o}s-Ginzburg-Ziv constant $s'_{\mathcal{L}}(G)$} defined as the smallest $\ell$ such that any \emph{zero-sum} sequence of length at least $\ell$ with elements in $G$ contains a zero-sum subsequence whose length is in $\mathcal{L}$. When $\mathcal{L} = \{t\}$ is a singleton set, we ignore the bracket notation. Note that one may also study the problem for subsets $G_0 \subseteq G$. However, in this paper we will always consider the modified or unmodified constant of the entire group $G$. In 2019, Berger and Wang determined modified EGZ constants in the finite cyclic case and some extensions. In particular they prove:
\begin{thm}[\cite{BW}, Theorem 1.3]\label{thm:BW1}
The modified EGZ constant of $\Z/n\Z$ is given by $s'_{nt}(\Z/n\Z) = (t+1)n - \ell + 1$, where $\ell$ is the smallest integer such that $\ell \nmid n$. 
\end{thm}
\begin{thm}[\cite{BW}, Theorem 1.4]\label{thm:BW2}
We have $s'_n\left((\Z/n\Z)^2\right) = 4n - \ell + 1$ where $\ell$ is the smallest integer such that $\ell \geq 4$ and $\ell \nmid n$. 
\end{thm}
They also state the following problem:
\begin{prblm}[\cite{BW}, Problem 4.3]\label{prblmone}
Compute $s'_{nt}\left((\Z/n\Z)^2\right)$ for $t > 1$. 
\end{prblm}

Our first two results provide partial answers for Problem \ref{prblmone}. 

\begin{thm}\label{thm:modegzprimes}
If $p \neq 3$ is prime and $t \geq 2$ then the modified EGZ constant of $\left(\Z/p\Z\right)^2$ is given by
\begin{equation*}
    s'_{pt}\left((\Z/p\Z)^2\right) = (t+2)p - 2.
\end{equation*}
If $p = 3$, we have
\begin{equation*}
    s'_{3t}\left((\Z/3\Z)^2\right) = 3(t+1).
\end{equation*}
\end{thm}
\begin{thm}\label{thmone}
Let $t \geq 2$ and write $n = pm$. Then the modified EGZ constant of $\left(\Z/n\Z\right)^2$ satisfies the bounds
\begin{equation*}
    (t+2)n - k + 1 \leq s'_{nt}\left((\Z/n\Z)^2\right) \leq (t+2)n + m - 3,
\end{equation*}
where $k$ is the smallest integer such that $k \geq 3$ and $\gcd(n,k) = 1$. Note that when $n = p$ is prime, the upper and lower bounds match and we obtain Theorem \ref{thm:modegzprimes}.
\end{thm}

In 2006, Halter-Koch and Geroldinger obtained the following result.
\begin{thm}[\cite{HG}, Theorem 5.8.3]\label{thm:mnegz}
The EGZ constant of $\left(\Z/n_1\Z \times \Z/n_2\Z\right)$, where $n_1 \mid n_2$, is given by
\begin{equation*}
    s_{n_2}\left(\Z/n_1\Z \times \Z/n_2\Z\right) = 2n_1 + 2n_2 - 3.
\end{equation*}
\end{thm}
We investigate the problem of computing the modified constant for this group first posed in \cite{BW}. 
\begin{prblm}[\cite{BW}, Problem 4.4]\label{prblmtwo}
Compute $s'_{n_2t}\left(\Z/n_1\Z \times \Z/n_2\Z\right)$ for $t \geq 1$ and $n_1 \mid n_2$. 
\end{prblm}

We give bounds for this case. We split it up into two theorems. In Theorem \ref{thmtwo}, we provide upper and lower bounds when $t = 1$. For $t > 1$, we are able to prove the following upper bound for the modified EGZ constant. 
\begin{thm}\label{thmtwo}
The modified EGZ constant of $\left(\Z/n_1\Z \times \Z/n_2\Z\right)$, where $n_1 \mid n_2$, satisfies the bounds 
\begin{equation*}
    2n_2 - \ell \leq s'_{n_2}\left(\Z/n_1\Z \times \Z/n_2\Z\right) \leq 2n_1 + 2n_2 - \ell + 1,
\end{equation*}
where $\ell$ is the smallest integer such that $\ell \geq 4$ and $\ell \nmid n_2$.
\end{thm}

\begin{thm}\label{thmthree}
Let $\ell$ be the smallest integer such that $\ell \geq 4$ and $\ell \nmid n_1, n_2$. Let $t \geq 1$ and $n_1 \mid n_2$. Then 
\begin{equation}
    s'_{n_2t}\left(\Z/n_1\Z \times \Z/n_2\Z\right) \leq 2n_1 + (t+1)n_2 - \ell + 1.
\end{equation}
\end{thm}

 Lastly, we investigate the (\emph{unmodified}) EGZ constant of $\left(\Z/n\Z\right)^2$. In 1983, Kemnitz \cite{K} conjectured that $s_{n}\left((\Z/n\Z)\right)^2 = 4n-3$. In 1993, Alon and Dubiner \cite{AD} proved that $s_{n}\left((\Z/n\Z\right)^2 \leq 6n-5$ and showed $s_{p}\left((\Z/p\Z)^2\right) \leq 5p-2$ for sufficiently large primes $p$. In 2000, R\'oyai \cite{Ro} proved $s_p\left((\Z/p\Z)^2\right) \leq 4p-2$. Finally, in 2007, Reiher \cite{R}  resolved Kemnitz's Conjecture. 
\begin{thm}[\cite{R}, Theorem 3.2]\label{thm:reiher}
If $J$ is a sequence of length $4n-3$ in $\left(\Z/n\Z\right)^2$, then $(n \mid J) > 0$. 
\end{thm}
We consider the EGZ constant when $\mathcal{L} = \{nt\}$, $t \geq 1$, i.e., the minimal $\ell$ such that any sequence of length $\ell$ contains a zero-sum subsequence of length $2n$ or $3n$, etc. 

\begin{thm}\label{thm:egzprimes}
If $t\geq2$ and $n = p$ is prime, then we have 
\begin{equation*}
    s_{pt}\left((\Z/p\Z)^2\right) = (t+2)p-2.
\end{equation*}
\end{thm}
\begin{cor}\label{cor:egznonprimeintro}
Let $t \geq 2$ and write $n = pm$. We have
\begin{equation*}
    (t+2)n - 2 \leq s_{nt}\left((\Z/n\Z)^2\right) \leq (t+2)n + m - 3.
\end{equation*}
Note if $n = p \neq 3$ is prime, then $m = 1$ and we recover Theorem \ref{thm:egzprimes}.
\end{cor}

\section{Proofs of Theorems \ref{thmone} and  \ref{thm:modegzprimes}}
In this section we give the proof of Theorem \ref{thmone}. As in \cite{BW}, if $J$ is a sequence of elements of $\left(\Z/n\Z\right)^2$, we use $(k \mid J)$ to denote the number of zero-sum subsequences of $J$ of size $k$. The proof closely follows the ideas presented in \cite{BW}. 

\begin{prop}\label{prop:gcdlowerbound}
Let $3 \leq k \leq n-1$ be the least integer such that $\gcd(n,k) = 1$. Then there exists a zero-sum sequence in $(\Z/n\Z)^2$ of length $(t+2)n - k$ which contains no zero-sum subsequences of length $nt$. 
\end{prop}
\begin{proof}
Consider a sequence of the form 
\begin{align*}
    (0,0) \ \ a &= tn-1 \\
    (1,0) \ \ b &= n - (k - 2) \\
    (0,1) \ \ c &= n - (k - 2) \\
    (1,1) \ \ d &= k - 3,
\end{align*}
where $a$ denotes the number of $(0,0)$'s, etc. It suffices to show that there is no zero-sum subsequence of any length among the nonzero elements, otherwise we could add copies of $(0,0)$ until we have a zero-sum subsequence of length $nt$. Indeed the sum of the nonzero elements is $(n-1, n-1)$ and there are at most $n-1$ nonzero elements with value $1$ being summed in each coordinate, so there is no zero-sum subsequence modulo $n$. Note also that since $k \geq 3$, $b,c \leq n-1$ and $d \leq n-4$, we cannot form a zero-sum subsequence using copies of only one basis element. We claim there exists $(r,s) \in (\Z/n\Z)^2$ such that adding $(r,s)$ to each term of the above sequence will result in a zero-sum sequence. Note that adding $(r,s)$ to each term does not change the fact that there is no zero-sum subsequence of length $nt$. Indeed, we only need to satisfy the divisibilty relations 
\begin{align*}
    (tn-1+n-k+2)r + (n - k + 2 + k - 3)(r+1) &\equiv 0 \pmod n \\
    (tn - 1 + n - k + 2)s + (n - k + 2 + k - 3)(s+1) &\equiv 0 \pmod n,
\end{align*}
which reduce to 
\begin{align*}
    k(-r) &\equiv 1 \pmod n \\
    k(-s) &\equiv 1 \pmod n.
\end{align*}
We can solve for $(r,s)$ since $\gcd(n,k) = 1$. 
\end{proof}

\begin{prop}\label{prop:p=3lowerbound}
There exists a zero-sum sequence in $(\Z/3\Z)^2$ of length $(3t+2)$ which contains no zero-sum subsequence of length $3t$. 
\end{prop}
\begin{proof}
Consider a sequence of the form 
\begin{align*}
    (0,0) \ & \ 3t - 1 \\
    (1,0) \ & \ 1 \\
    (0,1) \ & \ 1 \\
    (1,1) \ & \ 1.
\end{align*}
There is clearly no zero-sum subsequence of length $3t$. We claim there exists $(r,s) \in (\Z/n\Z)^2$ such that adding $(r,s)$ to each term of the above sequence will result in a zero-sum sequence. It is easy to check that $(2,2)$ works. 
\end{proof}
 Proposition \ref{prop:p=3lowerbound} provides the lower bound for the $p = 3$ case of  Theorem \ref{thm:modegzprimes}. Proposition \ref{prop:gcdlowerbound} provides the lower bound for both Theorem \ref{thm:modegzprimes} and Theorem \ref{thmone}, by noting that when $n=p$ is prime, $k = 3$.
 
\begin{prop}\label{prop:egz3upper}
If $J$ is a zero-sum sequence in $(\Z/3\Z)^2$ of length $3(t+1)$ then $(3t \mid J ) > 0$. 
\end{prop}
\begin{proof}
We induct on $t$. Note that when $t=2$, we have $3(t+1) = 9 = 4(3) - 3$.
By Theorem \ref{thm:reiher}, we can remove a zero-sum subsequence of length $3$, leaving us with $6$ elements. Since the original sequence of length $9$ was zero-sum, the remaining $6$ elements are zero-sum, so we have found our zero-sum subsequence of length $3t$. Now suppose the statement is true for all positive integers at most $t \geq 2$. Consider a zero-sum sequence of length $3((t+1) + 1)$. We have
\begin{equation*}
    3((t+1) + 1) \geq 4(3) - 3,
\end{equation*}
since $t \geq 1$. So we remove a zero-sum subsequence of length $3$. This leaves a zero-sum sequence with $3(t+1)$ elements. By the induction hypothesis, this has a zero-sum subsequence of length $3t$. Combining this subsequence with the zero-sum subsequence of length $3$ we removed yields a zero-sum subsequence of length $3(t+1)$, as desired. 
\end{proof}

Note that in general the modified EGZ constant is bounded above by the EGZ constant. If any sequence of some length has a zero-sum subsequence, then surely any zero-sum sequence of that same length will have a zero-sum subsequence. Theorem \ref{thm:egzprimes} and Corollary \ref{cor:egznonprimeintro} provide the upper bounds to finish the proofs of Theorems \ref{thm:modegzprimes} and \ref{thmone}. Note in the case $p=3$, the value of the upper bound provided by Theorem \ref{thm:egzprimes} is exactly one more than the length in Proposition \ref{prop:egz3upper}. 

Now we prove an analogue of a key lemma from \cite{BW}.
\begin{lem}[\cite{BW}, Lemma 3.4]\label{lem:BWlemma3.4}
If $J$ is a zero-sum sequence of length $3n$ in $(\Z/n\Z)^2$, then $(n \mid J)$. 
\end{lem}
\begin{prop}\label{prop:BWlemma3.4analogue}
If $J$ is a zero-sum sequence of length $3n$ in $(\Z/n\Z)^2$ then $(2n \mid J) > 0$. 
\end{prop}
\begin{proof}
By Lemma \ref{lem:BWlemma3.4}, $(n \mid J) > 0$. If $(n \mid J) = 1$, then the complement sequence of length $2n$ is zero-sum since $J$ is zero-sum. Otherwise $(n \mid J) \geq 2$, in which case we can pick $2$ of the zero-sum subsequences of length $n$ and combine them to obtain a zero-sum subsequence of length $2n$. 
\end{proof}
Now we generalize Proposition \ref{prop:BWlemma3.4analogue} for all $t \geq 2$. 
\begin{cor}\label{cor:modegzallt}
If $J$ is a zero-sum sequence of length $(t+1)n$ in $(\Z/n\Z)^2$ and $t \geq 2$, then $(tn \mid J) > 0$. 
\end{cor}
\begin{proof}
The $t=2$ case is Proposition \ref{prop:BWlemma3.4analogue}. Assume $t \geq 3$. Then $(t+1)n > 4n-3$. By Theorem \ref{thm:reiher}, we can remove zero-sum subsequences of length $n$ until there are exactly $3n$ remaining. This gives us $t-2$ zero-sum subsequences of length $n$. Since $J$ is zero-sum, the $3n$ remaining elements are zero-sum. Hence by Proposition \ref{prop:BWlemma3.4analogue}, there is a zero-sum subsequence of length $2n$. Combining this with the $t-2$ zero-sum subsequences of length $n$ gives a zero-sum subsequence of length $nt$. 
\end{proof}
\section{Proof of Theorems \ref{thmtwo} and \ref{thmthree}}

\begin{prop}\label{n1n2prop}
Let $\ell$ be the smallest positive integer greater than or equal to $4$  such that $\ell \nmid n_1$. If $J$ is a zero-sum sequence in $G = \left(\Z/n_1\Z \times \Z/n_2 \Z\right)$ with $n_1 \mid n_2$ and $J$ has length at least  $2n_1 + 2n_2 - \ell + 1$, then $(n_2 \mid J) > 0$.
\end{prop}

\begin{proof}
Assume $n_1 \neq n_2$, otherwise this is just the $(\Z/n\Z)^2$ case. We proceed by strong induction on the exponent of the group. Note that $\exp(G) = n_2$ in this case. Let $d$ be a divisor of $n_1$ such that $d \mid n_1$, $d < n_2$ and write $n_1 = dm_1$ and $n_2 = dm_2$. Note that $H = \left(\Z/m_1\Z \times \Z/m_2\Z\right)$ is a subgroup of $G$. When $\exp(G) = 2$, the claim is clearly true. Suppose the claim is true for all $\exp(G) < n_2$. First consider a zero-sum sequence of length $2n_1 + 2n_2 - d$. Note that $2n_1 + 2n_2 - d \geq 4d \geq 4d-3$, so by Theorem \ref{thm:BW2} we can remove subsequences of length $d$ with sum $0 \pmod d$ until there are exactly $3d$ remaining. Then by Lemma \ref{lem:BWlemma3.4}, we can break off another $d$ elements to obtain $2m_1 + 2m_2 - 3$ blocks of size $d$, with sums $dx_1, \dots, dx_{2m_1+2m_2-3}$, for some $x_i$. By the induction hypothesis, since \begin{equation*}
    2m_1 + 2m_2 - 3 \geq 2m_1 + 2m_2 - \ell +1, 
\end{equation*}
some $m_2$ of the $x_i$ must sum to $0$ in $(\Z/m_1\Z \times \Z/m_2\Z)$. Combining the corresponding blocks gives a subsequence of length $n_2$ whose sum is zero in $\left(\Z/n_1\Z \times \Z/n_2\Z\right)$. Now note that since $\ell$ is the least integer such that $\ell \nmid n_1$, we have $\ell-1 \mid n_1$. Since $n_1 \mid n_2$, we also have $\ell - 1 \mid n_2$. Letting $d = \ell - 1$ finishes the proof.
\end{proof}

Now we will show that if $|J|$ were any smaller, there couldn't be a zero-sum subsequence of length $n_2$. 
\begin{prop}\label{boundtight}
Suppose $4 \leq \ell \nmid n_2$. There exists a zero-sum sequence in $\left(\Z/n_1\Z \times \Z/n_2\Z\right)$ of length $2n_2 - \ell$ which contains no zero-sum subsequences of length $n_2$. 
\end{prop}

\begin{proof}
Let $g := \gcd(\ell,n_2)$. Consider a sequence of the form 
\begin{align*}
    (0,0) \ \ &a = n_2-\ell+g \\
    (1,1) \ \ &b = n_2 - g.
\end{align*}
It is easy to verify that this does not contain a zero-sum subsequence of length $n_2$. We claim there exists $(r,s) \in \left(\Z/n_1\Z \times \Z/n_2\Z\right)$ such that adding $(r,s)$ to each term will result in a zero-sum sequence. Note again that adding $(r,s)$ to each element won't change the fact that there is no zero-sum subsequence of length $n_2$.
Since $\ell \nmid n_2$, $g \leq \ell/2$. Therefore $a \leq n_2 - \ell/2  \leq n_2 - 2$, and $g \geq 1$, so $b \leq n_2-1$. To find $(r,s)$ we need only to satisfy the divisibility relations
\begin{align*}
    r(-\ell) &\equiv g \pmod{n_1} \\
    s(-\ell) &\equiv g \pmod{n_2}.
\end{align*}
By the definition of $g$, we can find solutions $(r,s)$ to make the sequence zero-sum. 
\end{proof}

Proposition \ref{n1n2prop} and \ref{boundtight} together imply Theorem \ref{thmtwo}. For the proof of Theorem \ref{thmthree}, we begin with the following corollary. 

\begin{cor}\label{cor:n2t}
Let $\ell$ be the smallest integer such that $\ell \geq 4$ and $\ell \nmid n_2$. Let $t \geq 1$ and $n_1 \mid n_2$. If $J$ is a zero-sum sequence in $\left(\Z/n_1\Z \times \Z/n_2\Z\right)$ of length at least $2n_1 + (t+1)n_2 - \ell + 1$, then $(n_2t \mid J) > 0$. 
\end{cor}
\begin{proof}
We proceed by induction on $t$. Note the base case $t = 1$ is given by Proposition \ref{n1n2prop}. Now suppose the statement is true for positive integers less than $t > 1$. Then $J$ contains a zero-sum sequence of length $(t-1)n_2$. Remove this sequence from $J$. Then $J$ has $2n_1 + 2n_2 - \ell + 1$ elements remaining, which sum to zero since $J$ was zero-sum. This reduces to the base case, so $J$ contains a zero-sum subsequence of length $n_2$. Combining this with the $(t-1)n_2$ length sequence gives a zero-sum subsequence of length $n_2t$. 
\end{proof}

Corollary \ref{cor:n2t} 
proves Theorem \ref{thmthree}.

\section{Proofs of Theorem \ref{thm:egzprimes} and Corollary \ref{cor:egznonprimeintro}}
\begin{prop}\label{prop:egzlowerbound}
Let $t \geq 1$ and $n \geq 2$. There exists a sequence in $\left(\Z/n\Z\right)^2$ of length $(t+2)n - 3$ which contains no zero-sum subsequence of length $nt$. 
\end{prop}
\begin{proof}
Consider the following sequence:
\begin{align*}
    (0,0) \ \ a &= tn - 1 \\
    (1,0) \ \ b &= n-1 \\
    (0,1) \ \ c &= n-1. 
\end{align*}
We clearly cannot make a sequence of $tn$ $(0,0)$'s. It suffices to verify that there does not exist a zero-sum subsequence of any length among the nonzero elements. Otherwise, we could just add enough $(0,0)$'s to get a zero-sum subsequence of length $tn$. Suppose we use $i$ $(1,0)$'s, and $j$ $(0,1)$'s, where $0 \leq i,j \leq n-1$. In order for the subsequence to be zero-sum, necessarily we would need 
\begin{equation*}
    i \equiv 0 \pmod{n} \text{ and } j \equiv 0 \pmod{n}.
\end{equation*}
Since $0 \leq i,j \leq n-1$, the only solution is $i = j=0$. Hence there is no zero-sum subsequence. 
\end{proof}

This gives the lower bound in both Theorem \ref{thm:egzprimes} and Corollary \ref{cor:egznonprimeintro}. 

To prove Theorem \ref{thm:egzprimes}, we will need the following preliminary lemma. 

\begin{lem}[\cite{R}, Corollary 2.3]\label{lem:reiherlemma}
Let $p$ be a prime, and let $J$ be a sequence of elements in $\left(\Z/p\Z\right)^2$. If $|J| = 3p-2$ or $|J| = 3p-1$, then 
\begin{equation*}
    1 - (p \mid J) + (2p \mid J) \equiv 0 \pmod p. 
\end{equation*}
\end{lem}

\begin{prop}\label{prop:egzprime}
If $J$ is a sequence in $\left(\Z/p\Z\right)^2$ of length $4p-2$, then $(2p \mid J) > 0$. 
\end{prop}
\begin{proof}
Note that $4p-2 > 4p-3$. By Theorem \ref{thm:reiher}, $J$ contains a zero-sum subsequence of length $p$. Removing the sequence from $J$, we are left with $3p-2$ elements. By Lemma \ref{lem:reiherlemma} we have
\begin{equation*}
    1 - (p \mid J) + (2p \mid J) \equiv 0 \pmod p.
\end{equation*}
If $(2p \mid J) > 0$, we're done and have found our zero-sum subsequence of length $2p$. Otherwise, $(2p \mid J) = 0$ which implies 
\begin{equation*}
    (p \mid J) \equiv 1 \pmod p.
\end{equation*}
Therefore, $(p \mid J) > 0$, so there is another zero-sum subsequence of length $p$. Combining this with the first one gives a zero-sum subsequence of length $2p$. 
\end{proof}

\begin{cor}\label{cor:egzprime}
Let $t \geq 2$. If $J$ is a sequence in $\left(\Z/p\Z\right)^2$ of length $(t+2)p - 2$, then $(tp \mid J) > 0$. 
\end{cor}
\begin{proof}
We proceed by induction on $t$. The case $t=2$ follows from Proposition \ref{prop:egzprime}. Suppose the statement is true for positive integers less than $t > 2$. Since $t \geq 2$, we have
\begin{equation*}
    (t+2)p-2 \geq 4p-3.
\end{equation*}
By Theorem \ref{thm:reiher} $J$ has a zero-sum subsequence of length $p$. Now remove the sequence so that $J$ has $((t-1) + 2)p - 2$ elements remaining. By the induction hypothesis, $J$ has a zero-sum subsequence of length $(t-1)p$. Combining this with the zero-sum subsequence of length $p$ yields a zero-sum subsequence of length $tp$.
\end{proof}
Proposition \ref{prop:egzlowerbound} and Corollary \ref{cor:egzprime} imply Theorem \ref{thm:egzprimes}.\\

Now we prove a version of Proposition \ref{prop:egzprime} for non-prime $n$. 
\begin{prop}\label{prop:egznonprime}
Write $n = pm$. If $J$ is a sequence in $\left(\Z/n\Z\right)^2$ of length $4n - 2 +(m-1)$, then $(2n \mid J) > 0$. 
\end{prop}
\begin{proof}
Note that $4n - 2 + (m-1) > 4m - 3$, so we can find some $m$ elements whose sum is $0 \pmod m$. Denote their sum by $mx_1$ and remove the elements from $J$. We can continue doing this until there are exactly $3m-3$ elements remaining. This gives us $4p-2$ blocks of size $m$ whose sums are $mx_1, \dots, mx_{4p-2}$ for some $x_i$'s. By Proposition \ref{prop:egzprime}, there is some $2p$ of the $x_i$'s summing to $0 \pmod p$. Combining the blocks gives us $2n$ elements whose sum is $0 \pmod n$. 
\end{proof}
\begin{cor}\label{cor:egznonprime}
Write $n = pm$ and let $t \geq 2$. If $J$ is a sequence in $\left(\Z/n\Z\right)^2$ of length $(t+2)n - 2 + (m-1)$, then $(tn \mid J) > 0$. 
\end{cor}
\begin{proof}
We induct on $t$. The $t = 2$ case is Proposition \ref{prop:egznonprime}. Suppose the statement is true for positive integers less than $t > 2$. Since $t \geq 2$, we have
\begin{equation*}
    (t+2)n - 2 + (m-1) \geq 4n - 3.
\end{equation*}
By Theorem \ref{thm:reiher}, $J$ has a zero-sum subsequence of length $n$. Removing it leaves us with $((t-1) + 2)n - 2 + (m-1)$ elements. By the induction hypothesis, we can remove a zero-sum subsequence of length $(t-1)n$. Combining these elements with the zero-sum subsequence of length $n$ yields a zero-sum subsequence of length $tn$, as desired. 
\end{proof}

Corollary \ref{cor:egznonprime} and Proposition \ref{prop:egzlowerbound} imply Corollary \ref{cor:egznonprimeintro}.
\section{Bounds for Modified EGZ Constants in $(\Z/p\Z)^d$}

\begin{prop}\label{primeprop}https://www.overleaf.com/project/5cf1d1637138370431c623d4
Let $p>3$ be prime, and let $J$ be a sequence of elements in $\left(\Z/p\Z\right)^3$. Then if $|J| = 4p - 4$, then 
\begin{equation*}
1 - (p-1 \mid J) - (p \mid J) + (2p-1 \mid J) + (2p \mid J) - (3p - 1 \mid J) - (3p \mid J) \equiv 0 \mod p.
\end{equation*}
\end{prop}

To prove the proposition, we use the following classical theorem. 
\begin{thm}[Chevalley-Warning]\label{CW}
Let $n, d_1, \dots, d_r$ be positive integers such that $d_1 + \dots + d_r < n$. For each $1 \leq i \leq r$ let $P_i(t_1, \dots, t_n) \in \F_q[t_1, \dots, t_n]$ be a polynomial of degree $d_i$ with zero constant term. Then there exists $0 \neq x = (x_1, \dots, x_n) \in \F_q^n$ such that $P_i(x) = 0$ for all $1 \leq i \leq r$. Furthermore, let 
\begin{equation}
    Z = \#\{x = (x_1, \dots, x_n) \in \F_q^n \  : \  P_1(x) = \dots = P_r(x) = 0\}.
\end{equation}
Then $Z \equiv 0 \mod p$. 
\end{thm}

Now we will prove Proposition \ref{primeprop}. 
\begin{proof}
Let $J = \{(a_n, b_n, c_n) : 1 \leq n \leq 4p-4$\}. Consider the following polynomials over $\F_p[t_1, \dots, t_{4p-3}]$:
\begin{align*}
    P_1(t) &= \sum_{i=1}^{4p-4} t_i^{p-1} + t_{4p-3}^{p-1} \\
    P_2(t) &= \sum_{i=1}^{4p-4} a_it_i^{p-1} \\
    P_3(t) &= \sum_{i=1}^{4p-4} b_it_i^{p-1} \\
    P_4(t) &= \sum_{i=1}^{4p-4} c_it_i^{p-1}. 
\end{align*}
Since $4p-3 > 4p-4$, by Theorem \ref{CW}, there exists $0 \neq x = (x_1, \dots, x_{4p-3})$ such that $P_1(x) = \dots = P_4(x) = 0.$ We partition the solutions according to $(x_1, \dots, x_{4p-4}, 0)$ and $(x_1, \dots, x_{4p-4}, \text{ nonzero})$. 

First we consider solutions of the form $(x_1, \dots, x_{4p-4}, 0)$. Let $I = \{1 \leq i \leq 4p-4 \ : \ x_i \neq 0\}$. Note that $x^{p-1} = 1$ if $x$ is nonzero and $x^{p-1} = 0$ if $x=0$. Then since $P_1(x) = \dots P_4(x) = 0$, we have
\begin{equation*}
    \sum_{i \in I} 1 + 0 = \sum_{i \in I} a_i = \sum_{i \in I} b_i = \sum_{i \in I} c_i \equiv 0 \mod p. 
\end{equation*}
Therefore, $|I| \equiv 0 \mod p$ and, since $0 < |I| \leq 4p-4$, we have $|I| = p, 2p$, or $3p$. Note that this set of solutions contains the zero solution, so the total number of solutions where $x_{4p-3} = 0$ is
\begin{equation*}
    1 + (p-1)^p(p \mid J) + (p-1)^{2p}(2p \mid J) + (p-1)^{3p}(3p \mid J).
\end{equation*}
Now we consider the set of solutions of the form $(x_1, \dots, x_{4p-4}, \text{ nonzero})$. In this case, define $I$ the same way and since $P_1(x) = \dots = P_4(x) = 0$, we have
\begin{equation*}
    \sum_{i \in I} 1 + 1 = \sum_{i \in I} a_i = \sum_{i \in I} b_i = \sum_{i \in I} c_i \equiv 0 \mod p.
\end{equation*}
Therefore, $|I| \equiv -1 \mod p$, and since $0 < |I| \leq 4p-4$, we have $|I| = p-1, 2p-1$, or $3p-1$. Thus the number of solutions is 
\begin{equation*}
    (p-1)^p(p-1 \mid J) + (p-1)^{2p}(2p-1 \mid J) + (p-1)^{3p}(3p-1 \mid J). 
\end{equation*}
Reducing modulo $p$ and combining these with the other set of solutions yields the result. 
\end{proof}

This proof leads us to the following corollary. 
\begin{cor}\label{primecor}
If $|J| = 4p-3, 4p-2$, or  $4p-1$ then 
\begin{equation*}
    1 - (p \mid J) + (2p \mid J) - (3p \mid J) \equiv 0 \mod p
\end{equation*}
\end{cor}

\begin{cor}
Suppose $J$ is a zero-sum sequence in $(\Z/p\Z)^3$ and $|J| = 4p$. Then $(p \mid J) > 0$ or $(2p \mid J) > 0$. 
\end{cor}

\begin{proof}
Let $x \in J$ be arbitrary. Suppose towards a contradiction that $(p \mid J) = 0$ and $(2p \mid J) = 0$. Then we must also have $(p \mid J -  \{x\}) = 0$ and $(2p \mid J - \{x\}) = 0$. Since $|J - \{x\}| = 4p-1$, by Corollary \ref{primecor}, $(3p \mid J - \{x\}) \equiv -1 \mod p$. So $(3p \mid J - \{x\}) > 0$. Since $J$ is zero sum, note that if there was a zero-sum subsequence of length $p$, its complement sequence of length $3p$ must also be zero-sum. In other words, 
\begin{equation*}
    (p \mid J) = (3p \mid J) \geq (3p \mid J - \{x\}) > 0,
\end{equation*}
contradicting $(p \mid J) = 0$. 
\end{proof}

Note that the preceding few results are amenable to the exact same methods for higher dimensions. In general, for $(\Z/p\Z)^d$, one would construct $(d+1)$ polynomials using the Chevalley-Warning method. This would yield the following: \\
If $|J| = (d+1)(p-1)$, then 
\begin{equation}
    1 + \sum_{k=1}^d (-1)^k\left((kp - 1 \mid J) + (kp \mid J)\right) \equiv 0 \mod p. 
\end{equation}
Furthermore, if $|J| = (d+1)p - m$ for some $1 \leq m \leq d$, then 
\begin{equation*}
    1 + \sum_{k=1}^d (-1)^k(kp \mid J) \equiv 0 \mod p.
\end{equation*}
Lastly, this would imply that if $J$ is zero sum and $|J| = (d+1)p$, then at least one of $(p \mid J), \dots, \left((d-1)p \mid J\right)$ is greater than zero. This leads us to the following corollary. 
\begin{cor}
Let $p$ be prime, $G = \left(\Z/p\Z\right)^d$, and $\mathcal{L} = \{p, 2p, \dots, (d-1)p\}$. Then 
\begin{equation*}
    s'_\mathcal{L}(G) \leq (d+1)p.
\end{equation*}
\end{cor}
\section{Open Problems}
In 1973, Harborth \cite{H} considered the problem of computing $s_n\left((\Z/n\Z)^d\right)$ for higher dimensions. In particular, he proved the following bounds. 
\begin{thm}[Harborth, \cite{H}]\label{thm:Harborth}
We have
\begin{equation*}
    (n-1)2^d + 1 \leq s_n\left((\Z/n\Z)^d\right) \leq (n-1)n^d + 1.
\end{equation*}
In general the lower bound is not tight, but Harborth showed we have equality for $n=2^k$. 
\end{thm}
In 2019, this was improved by Naslund resulting in the following bounds.
\begin{thm}[Naslund, \cite{N}]\label{thm:Naslund}
\begin{equation*}
    s_p(\F_p^n) \leq (p-1)2^p(J(p)\cdot p)^n,
\end{equation*}
where $J(p)$ is a constant satisfying $0.8414 < J(p) < 0.91837$. 
\end{thm}
In 2019, Berger and Wang made the following conjecture.
\begin{conj}[Conjecture 4.2, \cite{BW}]\label{conj:BW}
If $n=2^k$ and $d \geq 1$, we have
\begin{equation*}
    s'_n\left((\Z/n\Z)^d\right) = 2^dn - \ell + 1,
\end{equation*}
where $\ell$ is the smallest integer such that $\ell \geq 2^d$ and $\ell \nmid n$. 
\end{conj}
We make the following conjecture.
\begin{conj}
Let $n, t, d \geq 1$ be positive integers. We have
\begin{equation*}
    s'_{nt}\left(\Z/n\Z\right)^d \leq (t + 2^d - 1)n - \ell + 1,
\end{equation*}
where $\ell$ is the smallest integer such that $\ell \geq 2^d$ and $\ell \nmid n$. 
\end{conj}

We also have not determined the EGZ constant $s_{nt}\left((\Z/n\Z)^2\right)$ for non-prime $n$. 
\begin{prob}
Compute $s_{nt}\left((\Z/n\Z)^2\right)$ for non-prime $n$ and $t \geq 2$. 
\end{prob}
\section*{Acknowledgements}
This research was conducted at the University of Minnesota Duluth REU and was supported by NSF / DMS grant 1659047 and NSA grant H98230-18-1-0010. I would like to thank Joe Gallian for running the program and Aaron Berger for helpful conversations and encouragement.

\bigskip

\end{document}